\documentclass[reqno,oneside,12pt]{amsart}

\usepackage[T1]{fontenc}
\usepackage{times,mathptm,caption}
\usepackage{amssymb,epsfig,verbatim,xypic,xcolor, graphicx, enumerate, tikzsymbols, transparent}

\theoremstyle{plain}

\newtheorem{thm}{Theorem}[section]

\newtheorem{pro}[thm]{Proposition}
\newtheorem{lem}[thm]{Lemma}
\newtheorem{proposition-principale}[thm]{Proposition principale}
\newtheorem{thm-principal}{Main Theorem}

\theoremstyle{definition}

\newtheorem{eg}[thm]{Example}
\newtheorem{rem}[thm]{Remark}

\newenvironment{defi-G}
{\noindent{\bf Definition.}\it}{\\}

\newenvironment{thm-M}
{\noindent{\bf Main Theorem.}\it }{}

\newenvironment{thm-AA}
{\noindent{\bf Theorem A'.}\it}{\\ }

\newenvironment{thm-A}
{\noindent{\bf Theorem A.}\it}{\\ }

\newenvironment{thm-B}
{\noindent{\bf Theorem B.}\it} 

\newenvironment{thm-BB}
{\noindent{\bf Theorem B'.}\it}

\newenvironment{thm-C}
{\noindent{\bf Theorem C.}\it}

\def\vv{\vspace{0.1cm}}

\def\C{\mathbb{C}}
\def\R{\mathbb{R}}
\def\Q{\mathbb{Q}}
\def\H{\mathbb{H}}
\def\Z{\mathbb{Z}}

\newcommand{\norm}[1]{{\left\Vert#1\right\Vert}}
\newcommand{\abs}[1]{\left\vert#1\right\vert}

\def\jj{{\sf{j}}}

\def\T{{\mathrm{T}}}
\def\J{{\text{\sc{j}}}}

\def\P{\mathbb{P}}

\def\Aut{{\sf{Aut}}}

\def\Sym{{\sf{Sym}}}

\def\NS{{\mathrm{NS}}}

\def\Hyp{{\mathbb{H}}}

\def\vol{{\rm{vol}}}

\def\PGL{{\sf{PGL}}\,}

\def\GL{{\sf{GL}}\,}

\def\End{{\sf{End}}\,}

\def\Lat{{\mathrm{L}}}

\def\det{{\sf{det}}}

\newcommand{\Id}{{\rm Id}}

\def\Pic{{\mathrm{Pic}}}

\def\Sym{{\text{Sym}}}

\def\dist{{\sf{dist}}}

\newcommand{\serge}[1]{{\color{red}*}\marginpar{\tiny  \color{red} SC: #1}}

\addtocounter{section}{0}             
\numberwithin{equation}{section}       

\begin{document}

\setlength{\baselineskip}{0.56cm}        
%
%
\title[Parabolic automorphisms: Orbits and Betti maps]
{Parabolic automorphisms of hyperkähler manifolds: Orbits and Betti maps}
\date{2023/2024}
\author{Ekaterina Amerik and Serge Cantat}
\address{CNRS, IRMAR - UMR 6625 \\ 
Universit{\'e} de Rennes 
\\ France}
\address{Laboratory of Algebraic Geometry \\ 
Higher School of Economics\\
Moscow, Russia\\
also: IMO \\ Universit\'e Paris-Saclay\\
Orsay, France}

\email{ekaterina.amerik@gmail.com}
\email{serge.cantat@univ-rennes.fr}

%
%

%
%

%
%

\begin{abstract}
We study parabolic automorphisms of irreducible holomorphically symplectic manifolds with a lagrangian 
fibration.
Such automorphisms are (possibly up to taking a power) fiberwise translations on smooth fibers, and their orbits in a general fiber are dense (\cite{AV}).
We provide a simple proof that the associated Betti map is of maximal rank, in particular, the set of fibers where the induced translation is of finite order is dense as well.


\noindent{\sc{R\'esum\'e.}} Nous étudions les automorphismes paraboliques des variétés symplectiques holomorphes qui sont irréductibles et projectives. 

\end{abstract}

\maketitle

\setcounter{tocdepth}{1}

\tableofcontents

\vfill


{\small{ The work of the first author is partially supported by ANR project FANOHK and by HSE University Basic Research program. The research activities of the  second author  are supported by the European Research Council (ERC GOAT 101053021). }}

\pagebreak

\section{Introduction} 

\subsection{The dynamics of parabolic automorphisms}  

\subsubsection{} 
Let $X$ be an irreducible hyperk\"ahler (or ``holomorphic symplectic'') manifold of dimension $2g$.
This means that 
\begin{enumerate}[\rm (a)]
\item $X$ is a simply-connected compact K\"ahler complex manifold;
\item there is a holomorphic $2$-form $\sigma$ on $X$ which is symplectic, i.e. $\sigma^g$ is a non-vanishing holomorphic form of top degree; 
\item $\sigma$ is unique up to a nonzero multiplicative factor. 
\end{enumerate}

\subsubsection{} On $H^2(X, \Z)$ there is a non-degenerate integral quadratic form $q$ of signature $(3, b_2-3)$, the Beauville-Bogomolov form 
(see~\cite{huybrechts:survey2003}, \S 23.4). The signature of $q$ on $H^{1,1}(X; \R)$
is $(1, h^{1,1}(X)-1)$, so that the projectivization of the positive cone 
\begin{equation}
\{u\in H^{1,1}(X;\R)\; ; \; q(u,u)>0\}
\end{equation}
 can be viewed as a model of the hyperbolic space. We shall denote by $\Hyp_X$ this hyperbolic space, its dimension is $h^{1,1}(X)-1$. Its boundary $\partial \Hyp_X$ is the projectivization of the isotropic cone $\{u\in H^{1,1}(X;\R)\; ; \; q(u,u)=0\}$. 
 
 \subsubsection{} We denote by $\NS(X)$ the Néron-Severi group of $X$, 
 \begin{equation}
 \NS(X)=\H^{1,1}(X;\R)\cap \H^2(X;\Z).
 \end{equation}
  If $L$ is a line bundle on $X$, we denote by $[L]\in \NS(X)$ its Chern class.
 
 \subsubsection{}\label{par:loxo_para_elli} The group $\Aut(X)$ acts by isometries on $H^2(X; \Z)$ with respect to $q$ and preserves the Hodge decomposition, so that it acts also by isometries on $H^{1,1}(X;\R)$ and on $\Hyp_X$. As described in~\cite{Ratcliffe:book} for instance, there are three types of isometries of hyperbolic spaces, hence three types of automorphisms: elliptic, parabolic, and loxodromic. In this article, we study parabolic automorphisms. An automorphism $f$ of $X$ is  {\bf{parabolic}} if the induced automorphism $f^*$ of $H^{1,1}(X;\R)$ satisfies the following equivalent properties:
 \begin{enumerate}[\rm (a)]
 \item $f^*$ has exactly one fixed point on the boundary $\partial \Hyp_X$ and no fixed point in the interior; 
 \item  there is a positive iterate $(f^*)^n$ of $f^*$ acting as a unipotent matrix of infinite order on $H^{1,1}(X;\R)$ (resp.\ on $H^{2}(X;\Z)$);
 \item $\norm{(f^*)^n}= c(f) n^2+O(n)$ for some positive constant $c(f)$. (Here, $\norm{\cdot}$
 is any norm on $\End(H^{1,1}(X;\R))$ or $\End(H^2(X;\R))$.) 
 \end{enumerate} 
We refer to the Appendix for references and a proof of (c).

\subsubsection{} Let $f\in \Aut(X)$ be parabolic. Its fixed point on the boundary  $\partial\Hyp_X$ corresponds to a line in $H^{1,1}(X;\R)$ which is fixed pointwise by $f^*$; this line is integral: it is generated by some primitive isotropic class $\ell_f\in \NS(X)$. 
Moreover, the nef cone of $X$ being closed and $\Aut(X)$-invariant, we can choose $\ell_f$ to be the class of some nef line bundle. This uniquely determines $\ell_f$. 

Since $\Pic^0(X)=0$, there is a unique nef line bundle $L_f$ such that $[L_f]=\ell_f$, and then $f^*L_f=L_f$.

\subsubsection{}\label{par:lagrangian_conjecture} The so-called Lagrangian Conjecture (which seems to have been stated independently by several people, including  Hassett and Tschinkel, Huybrechts, and Sawon), also known as the Hyperk\"ahler SYZ Conjecture, says that a nef line bundle $L$ with $q([L],[L])=0$ should be semi-ample: this means that  $L^{\otimes n}$ should be base-point-free for large positive integers $n$. This conjecture has been  verified in all known examples and, applied to $L_f$, it says that the linear system of sections of $L_f^{\otimes n}$ defines a morphism 
\begin{equation}
p_f\colon X\to B 
\end{equation} 
with 
connected fibers of strictly positive dimension. According to Matsushita \cite{Mats}, such a morphism is a lagrangian fibration, which means that the smooth fibers of $p_f$ are lagrangian tori. 
The base $B$ of the fibration is a normal projective variety of Picard number $1$, which a priori can have quotient singularities. 

Then, there is an automorphism $f_B$ of $B$ such that 
\begin{equation}
p_f\circ f= f_B\circ p_f,
\end{equation} 
and it can be shown that $f_B$ has finite order  (\cite{lobianco:padic}). Thus, for some $k\geq 1$,
\begin{enumerate}[\rm (1)]
\item the action of $(f^k)^*$ on the $H^2(X;\Z)$ is unipotent, and of infinite order;
\item $p_f\circ f^k=p_f$  
\item $f^k$ acts as a translation on each smooth fiber (\cite{AV}, Proposition 3.8).
\end{enumerate}

\subsubsection{} Theorem 3.11 of  \cite{AV} shows that the orbits of $f^k$ must be dense in the euclidian topology on almost all smooth fibers of $p_f$. A natural question to ask is whether one nevertheless often encounters smaller orbit closures. For example, is the set of $b\in B$ such that $f^k$ is of finite order on $X_b$ (i.e.\ acts as a translation by a torsion element) dense in $B$? 
Our main theorem answers this question positively (see below for the definitions of translation vector and maximal variation).  

\vv

\begin{thm-A}
Let $X$ be an irreducible hyperk\"ahler manifold of dimension $2g$. Let $f$ be a parabolic automorphism of $X$ with an invariant fibration $p_f\colon X\to B$, and choose $k\geq 1$ such that $p_f\circ f^k=p_f$. Then, 

\vv
\noindent{\rm{(1)}} for any $p\in \{1, 2, \ldots, g\}$, there is a positive constant $c_p(f)$ such that
$$\norm{(f^n)^*}_{H^{p,p}(X;\R)}=c_p(f)n^{2p}+O(n^{2p-1});$$

\vv

\noindent{\rm{(2)}}  the translation vector of $f^k$ has maximal variation;

\vv
\noindent{\rm{(3)}} for any $s\in \{1, 2, \ldots, g\}$,
the subset of $B$ defined by 
$$
D_s(f^k) = \{b\in B\; ; \; {\text{ the closure of any orbit of $f^k_{\vert X_b}$ has dimension $s$ in $X_b$}}\} 
$$
is dense in $B$ for the euclidean topology.
\end{thm-A}

For instance, the following sets are dense in $B$: 
\begin{align}
 D_g &=\{b\in B\; ; \; {\text{every orbit of $f^k$ in $X_b$ is dense in $X_b$}}\}  \\
D_0 &= \{b\in B\; ; \; {\text{ $f^k_{\vert X_b}$ has finite order}}\}. \quad \quad\quad 
 \end{align} 
 
Note that we assume in Theorem~A that $f$ preserves a lagrangian fibration; as explained in Section~\ref{par:lagrangian_conjecture}, this is satisfied in all known examples. 

\subsubsection{} 

When $X$ is projective, Theorem~A is not new: it can be derived from results of Bakker, Gao, and Voisin. This is explained in Section~\ref{par:literature_gao}. Theorem~A has also been proven for all surfaces in~\cite{Cantat:Trans, Cantat-Dujardin:Transformation-Groups}, but it seems 
difficult to apply the same methods in higher-dimensional cases~({\footnote{The surfaces in \cite{Cantat-Dujardin:Transformation-Groups} are Kähler but do not have to be hyperk\"ahler. Indeed, if $X$ is a  compact complex surface, the intersection form defines a quadratic form on the second cohomology group of~$X$. If the surface is Kähler, its restriction to  $H^{1,1}(X;\R)$ is non-degenerate and of signature $(1,h^{1,1}(X;\R)-1)$. Thus, automorphisms of $X$ can also be classified into three types, elliptic, parabolic, or loxodromic. By a theorem of Gizatullin, every parabolic automorphism of a compact Kähler surface preserves a genus $1$ fibration (with finite order action on the base except when $X$ is a torus). }}).
The aim of this paper is to describe a new proof of it, and to extend the result to non-projective manifolds; on our way, we also extend a result of Lo Bianco (see Theorems~B and~C).

\subsection{Betti coordinates, translation vector, maximal variation}\label{par:betti_vector_variation}

\subsubsection{Betti coordinates}\label{par:intro_betti_coordinates} Let $p\colon X\to B$ be a fibration of a compact complex manifold. We shall always denote by $B^\circ$ the subset of regular values of $p$ where, by definition, the singularities of $B$ are put in $B\setminus B^\circ$. Suppose that for every $b\in B^\circ$, the fiber $X_b=p^{-1}(b)$ is a torus, isomorphic to $\C^g/\Lat(b)$ for some lattice $\Lat(b)\subset \C^g$.  

Let $U$ be a simply connected open subset of $B^\circ$, and $b_0$ a point of $U$. Suppose we have a holomorphic  section $s\colon U\to X$ of $p$.
 If one fixes a basis of $H_1(X_{b_0}; \Z)$, it can be propagated  continuously  to the fibers $X_b$ for $b\in U$ and this gives a trivialization $H^1(X_U;\Z)\simeq \Z^{2g}$. Then, there is a 
unique diffeomorphism 
\begin{equation}
\Phi\colon X_U\to U\times \R^{2g}/\Z^{2g}
\end{equation} 
such that 
\begin{enumerate}[\rm (i)]
\item  $p={\mathrm{pr}}_{U}\circ \Phi$, where ${\mathrm{pr}}_{U}$ is the projection from $U\times \R^{2g}/\Z^{2g}$ to $U$,
\item $\Phi\circ s_{\vert U}(b)=(b,0)$ for all $b\in U$, 
\item $\Phi\colon X_b\to \{b\}\times \R^{2g}/\Z^{2g}$ is an isomorphism of Lie groups for all $b\in U$,
\item $\Phi_*$ maps the basis of $H_1(X_U;\Z)$ to the canonical basis of $$\Z^g\simeq H_1(U\times \R^{2g}/\Z^{2g};\Z).$$ 
\end{enumerate}
This diffeomorphism is real analytic. We shall refer to $\Phi$ as the Betti diffeomorphism; points of $U\times \R^{2g}/\Z^{2g}$ can be written $(u,x)$ with $u$ in $U$ and $x=(x_1, \ldots, x_{2g})$ in $\R^{2g}$ (modulo $1$), and we shall refer to these as the {\bf{Betti coordinates}} (determined by $\Phi$). We refer to the triple  given by $U$, the section $s_{\vert U}$, and the basis of $H_1(X_U,\Z)$ as the {\bf{Betti datum}} used to define $\Phi$.

\subsubsection{The translation vector} Let $f$ be an automorphism of $X$ such that $p\circ f=p$ and $f$ acts by translations on the general fibers of $p$. Conjugating $f$ by $\Phi$ one gets a diffeomorphism of $U\times \R^{2g}/\Z^{2g}$ of type 
\begin{equation}
(u,x)\mapsto (u, x+ t_f(u))
\end{equation}
for some real analytic function $t_f\colon U\to \R^{2g}$ (or to $\R^{2g}/\Z^{2g}$). By definition, $t_f$ is the {\bf{translation vector}} of $f$ (in the Betti coordinates defined by $\Phi$). As we shall see in Section~\ref{par:translation_vectors}, the generic rank of $t_f$ is an even integer and this number does not depend on the choice of Betti coordinates. We shall refer to it as the {\bf{rank}} of the translation vector. The maximal possible rank is $\min(2\dim_{\C}(B), 2g)$. 
For lagrangian fibrations $\dim_{\C}(B)=g$ so in what follows we assume $\min(2\dim_{\C}(B), 2g)=2g$ for simplicity. Then,  we say that $t_f$ has {\bf{maximal rank}}, or equivalently that $t_f$ has {\bf{maximal variation}},  if its generic rank is $2g$. Lemma~\ref{lem:open_mapping_translation} shows that the variations of $t_f$ are maximal if and only if the image of $t_f$ is open in $\R^{2g}$, if and only if $t_f$ is an open mapping.

This explains the meaning of Assertion~(2) in Theorem~A and shows that this assertion implies Assertion~(3) (see Section~\ref{par:maximal_variation} for more on $t_f$ and a detailed proof of how (3) is derived from (2)). 

\subsubsection{} Now, suppose that $s\colon B\to X$ is a global holomorphic section of $p$. For every $b\in B^\circ$, we can declare that $s(b)$ is the neutral element of $X_b$ and, doing so, $X_b$ becomes a commutative complex Lie group. 

Then, $f\circ s$ is a new section of $p$, and the action of $f$ on $X_b$ is the translation by $f\circ s(b)-s(b)$ for every $b\in B^\circ$. Let us now set $t=f\circ s$ and forget about~$f$. The dynamical properties of $f$ can be translated into properties of $t$. More precisely, consider a Betti diffeomorphism $\Phi$ (determined by some choice of Betti datum) and set
\begin{equation}
t_f={\mathrm{pr}}_{\R^{2g}/\Z^{2g}}(\Phi\circ t)\mod \Z^{2g}.
\end{equation} 
This map $u\in U\mapsto t_f(u)={\mathrm{pr}}_{\R^{2g}/\Z^{2g}}(\Phi(t(u)))\mod \Z^{2g}$ is usually called the {\bf{Betti map}} associated to $t$ (and the chosen Betti datum). Thus, {\emph{the translation vector of $f$ has maximal variation if and only if the Betti map is generically of maximal rank $2g$}} 
. This property of the Betti map has been studied a lot, at least in the case when $X$ is projective, as explained below.


\subsection{General results on Betti maps}\label{par:literature_gao}

Let us explain how Theorem~A can be 
derived from  works of Gao, Voisin and Bakker when $X$ is projective.

\subsubsection{}  
The series of papers~\cite{ACZ, Gao1, Gao2} studies in a systematic way the Betti maps for fibrations in arbitrary dimension. They rely on theorems from functional transcendance theory, notably André’s theorem concerning the 
independence of abelian logarithms \cite{andre:1992} and the Ax–Schanuel theorem from~\cite{mok-pila-tsimerman}. 

Let us focus on \cite{Gao1} and \cite{Gao2}, since they contain optimal results regarding the variations of the Betti maps (i.e.\  of translation vectors). 
The tools used in \cite{Gao1} being somewhat simpler, we base our explanation on it and explain how it is related to Theorem~A. 

\subsubsection{}   
In \cite[Theorem 1.3]{Gao1}, Gao considers an abelian scheme $p\colon {\mathcal A}\to S$ of relative dimension $g$ over a smooth complex algebraic variety $S$, with a section $\xi$ (or more generally a multisection) generating ${\mathcal A}$.  
He proves that the associated Betti map is generically of rang $2g$ if the following three properties are satisfied:
\begin{enumerate}[\rm (a)]
\item the modular map $\mu: S\to {\mathcal A}_g$ is quasi-finite~(\footnote{Here, ${\mathcal A}_g$ is the space of polarized abelian varieties in dimension $g$ with respect to some polarization type and some level structure; since they are not relevant, we simply write ${\mathcal{A}}_g$.}); 
\item  $\dim(S)\geq g$, and 
\item the geometric generic fiber of the family is simple~(\footnote{See also \cite{Gao2},
where it is established that the non-maximality of the rank of the Betti map associated to a generating $\xi$ implies the existence of a quotient abelian scheme of low variation. For simplicity's sake we prefer to keep 
\cite{Gao1} as our main reference: it is almost equally quick to get applications to hyperk\"ahler manifolds 
from \cite{Gao1}, see \cite{Bak}.}). 
\end{enumerate}
To apply this result to our context, we can take $S$ to be the set of regular values $B^\circ$ of $p_f$. 
Then (b) is satisfied by construction.  Moreover, a  recent theorem of Bakker proving Matsushita's conjecture states that either $\mu$ is quasi-finite on a dense open subset of $S$ or $\mu$ is constant, i.e.\ the family is isotrivial  (see \cite{Bak}). 
Thus, in the non-isotrivial case, Gao's theorem is close to establish Assertion~(2) of Theorem~A; nevertheless, there is a subtlety here: in general it is not true that the geometric generic fiber is simple, though the scheme-theoretic generic fiber is simple and has Picard number $1$ (see \cite{Oguiso-Picard}). 

\begin{eg} Let $Y$ be a  K3 surface with a genus one fibration $h:Y\to \P^1$.
Set $X=Y^{[2]}={\mathrm{Hilb}}^2(Y)$. Then $X$ is hyperkähler and is fibered over $\P^2={\mathrm{Sym}}^2(\P^1)$:  the fiber over $a+b$, $a\neq b$, is $h^{-1}(a)\times h^{-1}(b)$. The generic fiber of this fibration $h^{[2]}: X\to \P^2$ ceases to be simple after a degree two extension of the function field of $\P^2$, corresponding to the map $(a,b)\mapsto a+b$ from $\P^1\times \P^1$ to $\P^2$. \end{eg}

For most applications, though, Gao's theorem 
works with some extra argument: see for example \cite[Corollary 9]{Bak} for a density statement similar to what we discuss here. 

Finally, coming back to the setting of  hyperk\"ahler manifolds, the case of isotrivial lagrangian fibrations is covered in a paper by Voisin \cite{Voisin-torsion} together with the case $\dim(X)\leq 8$ (i.e. $g\leq 4$).

\begin{rem}  In an earlier paper by Andr\'e, Corvaja, and Zannier 
\cite{ACZ}, the authors raise the question whether the Betti map associated to a section $\xi$ is generically of rank $2g$ under milder conditions:
$\Z\xi$ is Zariski dense in $\mathcal A$ (that is, $\xi$ generates $\mathcal A$), $\mathcal A$ has no fixed part over any \'etale finite covering 
of $S$,  $\mu$ is quasifinite, and (as above) $\dim(S)\geq g$. In \cite{ACZ}, the positive answer to this question is obtained in  dimension $g\leq 3$, and also in all dimensions under the additional assumption that the abelian scheme has no non-trivial endomorphism over any finite covering of the base. However in \cite{Gao2}, there is a counterexample for $g=4$ (Example 9.4). Clearly, this example is not hyperk\"ahler.
\end{rem}

\subsection{Strategy of proof} 
With the previous results in mind, the reason why we wrote this text is twofold. 
Firstly, Theorems~A and~B now hold uniformly, for projective and non-projective hyperkähler manifolds, and for isotrivial and non-isotrivial fibrations. 
Secondly, the proof follows a new route.
Gao obtains his result as a consequence of mixed Ax-Schanuel theorem. 
On one side, our argument is simpler because it relies on more basic principles; on the other side it applies only to the hyperkähler case, because we rely on Verbitsky's theorem~\cite{verbitsky:cohom}, Theorem 1.5, on the cohomology of hyperkähler manifolds (this is used to get Assertion~(1) of Theorem~A, which is -- in turn -- used to derive Assertion~(2)). 

The proof is done for projective hyperkähler manifold first, and then generalized  to the non-projective case. The argument for this last step is of independent interest and applies recent results of Soldatenkov and Verbitsky. 

\subsection{Acknowledgements} We are grateful to Thomas Gauthier, Misha Verbitsky, and Claire Voisin for useful discussions. We thank Pietro Corvaja, Andrey Soldatenkov, and Umberto Zannier for interesting feedback.

\section{Hyperkähler manifolds} 

 \vspace{0.2cm}
\begin{center}
\begin{minipage}{12cm}
{\sl{ 
In this section, $X$ is a hyperkähler manifold of dimension $2g$ with a holomorphic symplectic form $\sigma$, as in the introduction, and $q$ is the Beauville-Bogomolov form.  }}
\end{minipage}
\end{center}
\vspace{0.2cm}

\subsection{The Néron-Severi group}\label{par:neron-severi-group}
We denote by $q_X$ the restriction of $q$ to $\NS(X)$. If $A$ is a ring, we set $\NS(X; A)=\NS(X)\otimes_\Z A$, hence $\NS(X)=\NS(X; \Z)$.
When $X$ is projective, there are classes $u$ in $\NS(X)$ with $q_X(u,u)>0$, for instance Chern classes of ample line bundles. Conversely, a theorem of Huybrechts shows that if such a class $u\in \NS(X)$ exists, then $X$ is projective (see~\cite{huybrechts:inventiones, huybrechts:inventiones-correction}). 

 There are three possibilities for the signature of $q_X$ on $\NS(X; \R)$
\begin{enumerate}[\rm (a)]
\item $q_X$ is non-degenerate of signature $(1, \rho(X)-1)$; 
\item $q_X$ is degenerate with one-dimensional kernel, and takes only non-positive values; in this
case, following Oguiso (see~\cite{oguiso}, page 167), we say that $q_X$ is {\bf{parabolic}};
\item $q_X$ is negative definite. 
\end{enumerate}
The second and third cases do not appear when $X$ is projective.

\subsection{The transcendental lattice}\label{par:transcendental_lattice} The transcendental lattice $\T(X)$ is, by definition, the orthogonal complement of $\NS(X)$ in $H^2(X;\Z)$. 
The Lefschetz theorem on $(1,1)$-classes implies that $\T(X)$ is the smallest subgroup of $H^2(X; \Z)$ such that $\C\sigma$ is contained in $\T(X)\otimes_\Z\C$ and $H^2(X;\Z)/\T(X)$ is torsion free. 

\subsection{Fibrations and polarizations}\label{par:lagrangian-fibration}
 Let $p \colon  X\to B$ be a holomorphic fibration, that is, a proper surjection with connected fibers and $\dim(X)>\dim(B)\geq 1$. Then, $p$ is a lagrangian fibration: 
 \begin{enumerate}[\rm (1)]
 \item its fibers are projective, and the generic fiber is an abelian variety of dimension $g=\dim(X)/2$ on which $\sigma$ vanishes (see \cite{campana:isotrivialite}, Proposition 2.1, which the author attributes to Voisin); 
 \item the base $B$ is projective too, indeed it is K\"ahler and Moishezon with rational singularities (see \cite{kamenova-lehn} Theorem 2.8, \cite{namikawa:2002} Corollary 1.7). Moreover $B$ is $\Q$-factorial with Picard number 1 (\cite{matsushita:1999})\footnote{It is generally expected that $B$ is $\P^g$; when $B$ is smooth, this is a theorem by Hwang.}.
 \end{enumerate}

 Now, set $X^\circ=X_{B^\circ}$, where $B^\circ$ is defined as in Section~\ref{par:intro_betti_coordinates}. The projection $p\colon X^\circ \to B^\circ$ is a proper submersion, the fibers of which are naturally polarized abelian varieties. Indeed, the restriction homomorphism 
\begin{equation}
H^2(X;\Z)\to H^2(X_b;\Z)
\end{equation}
has  cyclic image (this remark has been first made by Oguiso in ~\cite{Oguiso-Picard}, see for example~\cite{AV} for a self-contained proof); let $R_b\subset H^2(X_b;\Z)$ denote this cyclic group.   If $\kappa$ is a Kähler form on $X$, then there is a  unique positive scalar multiple $\alpha\kappa$ such that $R_b$  is generated by the class  $[\alpha\kappa_{\vert X_b}]$; this  integral class gives a natural polarization of $X_b$ for each $b\in B^\circ$. When $X$ is projective, we can assume that $[\alpha\kappa]$ is in $\NS(X)$. 
 
 \subsection{Automorphisms}\label{par:automorphisms}

Let $p\colon X\to B$ be a lagrangian fibration on $X$. Since $\rho(B)=1$, there is a unique 
primitive ample class $h_B$ in $\NS(B)$; we denote by $h\in \NS(X)$ its pull back by $p$: this class is nef and isotropic (i.e. $q_X(h,h)=0$). 

Let $f$ be an automorphism that preserves the class $h$. Then $f$ preserves the fibration $p$ in the following sense: there is an automorphism $f_B$ of $B$ such that $p\circ f = f_B\circ p$. 
The automorphism $f_B$ preserves $h_B$. One can also find an embedding $B\subset \P^N(\C)$ 
such that $f_B$ preserves the Fubini-Study form (restricted to $B$); we denote such a  form by $\kappa_B$:
\begin{equation}
f_B^*\kappa_B = \kappa_B.
\end{equation}
The existence of such a form $\kappa_B$ is due to Lo Bianco (see~\cite{lobianco:padic}, Lemma 3.1). Let us sketch his proof. Since the Picard group of $B$ is cyclic and $B$ is projective, 
there is an $f_B$-invariant and very ample line bundle $L_B$ on $B$. Then, $f_B$ induces a linear transformation $F_B$ of $H^0(B;L_B)$ and the Kodaira-Iitaka embedding $\iota\colon B\to \P(H^0(B;L_B)^\vee)$ is equivariant: $\iota\circ f_B=F_B\circ \iota$. On the other hand, the volume form $\vol_X:=(\sigma\wedge \overline{\sigma})^g$ induces a probability measure $\mu_B=p_*\vol_X$ on $B$ with full support which is $f_B$ invariant. Then, the invariance of $\iota_*\mu_B$ under $F_B$, the fact that $\iota(B)$ is not contained in a hyperplane of  $\P(H^0(B;L_B)^\vee)$, and the equality ${\mathrm{Supp}}(\iota_*\mu_B)=\iota(B)$ imply that $F_B$ is contained in a compact subgroup of $\PGL(H^0(B;L_B)^\vee)$. Thus, up to a linear conjugacy, $F_B$ preserves the Fubini-Study metric.

A priori $B$ can have singularities, but it does make sense to speak of such a differential form as the restriction of a form defined on the ambient space $\P^N(\C)$; this is compatible with the definitions of Varouchas as suggested in \cite{soldatenkov-verbitsky}, Remark 2.2 (see also~\cite{demailly:1985}).

\vspace{0.2cm}

\begin{thm-B} Let $p\colon X\to B$ be a lagrangian fibration of an irreducible hyperkähler manifold $X$. 
Let $f$ be an automorphism of $X$ such that $p\circ f=f_B\circ p$ for some $f_B\in \Aut(B)$. Then there is an integer $k\geq 1$ such that 
\begin{enumerate}[\rm (1)]
\item $f^k$ preserves the symplectic form $\sigma$, i.e.\ $(f^k)^*\sigma=\sigma$;
\item $f_B^k=\Id_B$, i.e.\ $f^k$ preserves each fiber of $p$. 
\end{enumerate}
\end{thm-B}

If $X$ is not projective, then $f^*\sigma=\sigma$; in other words, one can take $k=1$ in the first assertion; it is implied by $h\in \T(X)$, see Theorem 2.4 in \cite{oguiso}. 

\begin{proof} 

Let us prove Assertion~(1) (see also~\cite{cantat:milnor, oguiso}). Since $f$ is parabolic, all eigenvalues of $f^*$ on $H^*(X;\C)$ have modulus $\leq 1$. Since $f^*$ preserves the lattice $H^*(X;\Z)$, its characteristic polynomial is
a monic polynomial with integer coefficients. Thus, the eigenvalues of $f$ are roots of unity. On the other hand, $\sigma$ is unique up to a scalar factor, so $f^*\sigma=\alpha\sigma$ where $\alpha$ is the eigenvalue of $f^*$ on $H^{2,0}(X;\C)$. If $k$ denotes the order of $\alpha$, then $(f^k)^*\sigma=\sigma$. 

When $X$ is projective, Assertion~(2) is part of a theorem of Lo Bianco in~\cite{lobianco:padic}(\footnote{Assertion (2)  has been claimed already in \cite{AV}, with an explanation that Lo Bianco's argument for the projective case was valid in general. That explanation does not seem to be correct, this is why we provide a proof of (2) based on Verbitsky's idea  in Section~\ref{par:proof-lobianco-non-projective}.}). The non-projective case reduces to the projective one as follows: using the form $\kappa_B$ one defines a family of new complex structures $X_t$ on $X$ (the ``degenerate twistor deformations'' studied by Verbitsky and Soldatenkov), all fibered over $B$, such that the map $f$ remains holomorphic on each $X_t$. Since some of these complex structures are projective, the conclusion will follow from the projective case.
Details are provided in Section~\ref{par:proof-lobianco-non-projective}.
\end{proof}
 
\section{Maximal variation and Betti coordinates} \label{par:maximal_variation}

 \vspace{0.2cm}
\begin{center}
\begin{minipage}{12cm}
{\sl{ 
In this section, we study the variations of the translation vector of a parabolic automorphism of a projective hyperkähler manifold. }}
\end{minipage}
\end{center}
\vspace{0.2cm}

\subsection{The setting}\label{par:setting_section_betti} 
We suppose that $X$ is a projective hyperkähler manifold, with a parabolic automorphism $f$ that preserves a lagrangian fibration $p_f\colon X\to B$, and use the notations introduced in the previous sections.  For simplicity, 
we replace $f$ by a positive iterate to assume that 
\begin{equation}
p_f\circ f=p_f \; \text{ and } \; f^*\sigma=\sigma
\end{equation} 
as in Theorem~B.

Since $X$ is projective, we can find a multisection $S$ of $p_f$. That is, $S\subset X$ is a smooth, irreducible, $g$-dimensional  subvariety of $X$ which is generically transverse to $p_f$. Then, $p_{f\vert S}\colon S\to B$ is generically finite. Moreover, if $U$ is a sufficiently small, non-empty, open subset of $B^\circ$,
 we can find such a multisection  that  is everywhere transverse to $p_f$ above $U$, i.e.\ $p_{f\vert S \cap {p_f}^{-1}(U)}$ is a non-ramified cover from $S\cap p_f^{-1}(U)$ to $U$ of some degree $d\geq 1$. 
With such a choice, and if $U$ is  simply-connected,  there are $d$ sections $s_i$ of $p_f$ above $U$ such that $S\cap U$ is the disjoint union of the $s_i(U)$. The degree $d$ is the intersection number $([S]\cdot [X_b])$ for any fiber $X_b$. 

\subsection{Translation vectors}\label{par:translation_vectors}
Let $U\subset B^\circ$ be simply connected and let $s\colon U\to X$ be a holomorphic section of $p_f$ above $U$. Fix a basis of $H_1(X_U; \Z)$ and consider the Betti diffeomorphism $\Phi$ and the translation vector $t_f$ associated to these data. 

\begin{lem}\label{lem:open_mapping_translation}
The following properties are equivalent.
\begin{enumerate}[\rm (1)]
\item $t_f\colon U\to \R^{2g}$ is an open mapping; 
\item $t_f(U)$ contains an open subset of $\R^{2g}$;
\item $t_f\colon U\to \R^{2g}$ has maximal rank $2g$ in the complement of a proper, real analytic subset of $U$.
\end{enumerate}
If they are not satisfied  the generic rank of $t_f$ is even and $\leq 2g-2$.
\end{lem}

We just sketch the proof of this lemma because it is already proven in~\cite{ACZ, Gao1, Cantat-Dujardin:Transformation-Groups}.  The first remark is that the fibers of the Betti projection $\pi_2\circ \Phi\colon X_U\to \R^{2g}/\Z^{2g}$ are complex submanifolds of $X_U$. The second remark is that,  viewed in Betti coordinates,  $t_f$ is just the projection of $t=f\circ s - s$ on $\R^{2g}/\Z^{2g}$ and $t$ is a holomorphic function (its differential $Dt$ intertwines the complex structure $\jj_U$ on $TU$ with the complex structure $\jj_X$ on $TX$).  
With this at hand, the first consequence is that $t_f$ is real analytic, and in particular the maximum of the rank of $(Dt)_u$, $u\in U$, is attained on the complement of a proper real analytic subset of $U$. Then, let $\jj(u)$  be the (translation invariant) complex structure  on $\R^{2g}/\Z^{2g}$ (or equivalently $\R^{2g}$) obtained from the restriction of $\jj_X$ to $X_u$ via $\Phi$:
\begin{equation}
\jj(u)(v)=\Phi_*(\jj_X(\Phi^{-1}_*v))
\end{equation}
 for every  vector  $v$  tangent to $\R^{2g}/\Z^{2g}$.
Then $(\jj(u))$  is a real analytic family of  complex structures and the second consequence is $(Dt_f)_u\circ \jj_U=\jj(u)\circ (Dt_f)_u$ for every $u\in U$. Thus, 
the generic rank of $t_f$ is even. These properties directly imply the lemma.

\begin{lem} The property 
``{\emph{$t_f$ has maximal variation}}''
does not depend on the Betti datum chosen to define the Betti coordinates. 
\end{lem}

Indeed, changing the section $s$ does not change $t_f$, and changing the basis of $H_1(X_U;\Z)$ changes $t_f$ into $A\circ t_f$ for some $A\in \GL_{2g}(\Z)$, so in both cases the property ``$t_f(U)$ contains an open subset of $\R^{2g}$'' is preserved by such a change. 

To show that the property does not depend on the choice of $U$, note that if $U\cap U'$ is non-empty, then Lemma~\ref{lem:open_mapping_translation}(3)  shows that $t_f$ has maximal variation on $U$ if and only if it has maximal variation on $U\cap U'$, and then this property propagates to $U'$. Then use that $B^\circ$ is connected.

\subsection{Volumes and variations} To study the variations of $t_f$, we shall rely on the following volumic characterization of its maximal variation. If $\kappa$ is a Kähler form on $X$, and if $W$ is a complex analytic subset of (some open subset of) $X$ of dimension $m$, its volume with respect to $\kappa$ is equal to 
\begin{equation}
\vol_\kappa(W)=\int_W \kappa^m.
\end{equation}
If $W\subset X$ is closed, its volume can be computed in cohomology as the intersection product $\vol_\kappa(W)=([W]\cdot [\kappa]^m)$, where $[\kappa]$ is the class of $\kappa$ and $[W]$ is the Poincaré dual of the homology class of $W$.

\begin{lem}\label{par:local_volume_estimate} Let $U$ be an open subset of $B$ such that its closure is contained in  $B^\circ$.
Let $\kappa$ be a Kähler form on $X$.  Let $M$ be a multisection of $p_f$. Then
$\vol_\kappa(f^n(M)\cap X_U)= O (n^{2g})$, and the following properties are equivalent
\begin{enumerate}[\rm (a)]
\item  $t_f$ does not have maximal variation;
\item $\vol_\kappa(f^n(M)\cap X_U)= O (n^{2g-1})$
as $n$ goes to $+\infty$;
\item $\norm{(f^n)^*\kappa^g}_{X_U}=O(n^{2g-1})$ as $n$ goes to $+\infty$, 
where $\norm{\cdot}_{X_U}$ denotes the uniform norm on $X_U$ (for sections of $\wedge^{g,g}(T X)$).
\end{enumerate}
\end{lem}

Instead of $O(n^{2g-1})$, we could obtain $O(n^{r})$ where $r$ is the generic rank of $t_f$ on $U$, but the weakest estimate will be sufficient for our purpose. This lemma implies again that maximal variation of the translation vector is an intrinsic property that does not depend on the choice of Betti coordinates. 

\begin{proof} 
By compactness of $\overline{U}\subset B^\circ$, we reduce to the case when $U$ is a ball (viewed in some local chart of $B^\circ$ containing $\overline U$,  $U$ is a ball in $\C^{g}$).

As in Section~\ref{par:translation_vectors}, $s\colon U\to X$ is a section of $p_f$ above $U$,  $\Phi$ is a Betti diffeomorphism and $t_f$ is the translation vector in the Betti coordinates; we set $S=s(U)$.
We can assume moreover that $s$ (resp.\ $\Phi$) extends to a neighborhood of $\overline U$ (resp. $p_f^{-1}(\overline U)$).  
We transport the riemannian metric associated to $\kappa$ by $\Phi$ to get a riemannian metric $\norm{\cdot}_{\kappa, \Phi}$ on $U\times \R^{2g}/\Z^{2g}$. Let $\norm{\cdot}_{euc}$ be the euclidean metric on $\C^{g}\times \R^{2g}$, restricted to $U\times \R^{2g}$.
Since $\overline{U}\subset B^\circ$, there is a  constant $A\geq 1$ such that  $A^{-1} \norm{\cdot}_{euc} \leq \norm{\cdot}_{\kappa, \Phi}\leq A \norm{\cdot}_{euc}$ uniformly on the tangent space of $U\times \R^{2g}/\Z^{2g}$. Thus, when estimating volumes, we can work with the usual euclidean metric in the Betti coordinates.

Let $d$ be the degree of the multisection $M$. 
Let $C\subset U$ be the branch locus of $p_{f\vert M}\colon M\cap p_f^{-1}(U)\to U$. 
Let $D$ be a real analytic subset of $U$ containing $C$ such that $U':=U\setminus D$ is simply connected. Since the Lebesgue measure of $D$ vanishes, the volume of $f^n(M)$ above $U'$ is the same as its volume above $U$.
But over $U'$, $M$  is a union of $d$ sections, so without loss of generality we may replace $M$ by one of them and assume that $M$ is in fact a section.  

In the Betti coordinates, $f$ becomes 
\begin{equation}
f_\Phi\colon (u,x)\mapsto (u, x+t_f(u))
\end{equation} 
$\Phi(S)$ is parametrized  by $u\mapsto (u,0)$, and $\Phi(M)$ by $u\mapsto (u,m^\Phi(u))$ for some real analytic function $m^\Phi$. Thus, $\Phi(f^n(M))$ is parametrized by 
\begin{equation}
u\mapsto (u, m^\Phi(u)+nt_f(u)),
\end{equation}
and the question is to estimate the volume of this submanifold of real dimension $2g$ with respect to the euclidean metric. 


In the tangent space of  $U\subset \C^g=\R^{2g}$, denote by $(v_i)_{i\leq 2g}$ the standard orthonormal basis. In the tangent space of $\R^{2g}/\Z^{2g}$,  denote by
$(e_i)_{i\leq 2g}$ the standard orthonormal  basis. The image of $v_i$ by the differential of $u\mapsto (u, m^\Phi(u)+nt_f(u))$ is the vector 
\begin{equation}
w_i(n;u)=v_i+ (Dm^\Phi)_u(v_i)+ n (Dt_f)_u(v_i).
\end{equation} 
Let us write $t_f(u)=(t_1(u), \ldots, t_{2g}(u))$ and 
$(Dt_f)_u(v_i)=(\partial_i t_j(u))$; similarly, $m^\phi(u)=(m^\Phi_1(u), \ldots, m^\Phi_{2g}(u))$ and  
$(Dm^\phi)_u(v_i)=(\partial_i m^\phi_j(u))$. We see that   the exterior product $w_1(n; u)\wedge \cdots \wedge w_{2g}(n; u)$ is a polynomial in $n$ of degree at most $2g$ with coefficients which are uniformly bounded, analytic functions of~$u$. For instance, when $g=1$ we obtain 
\begin{align*}
w_1  = \; & v_1+ \partial_1 m^\Phi_1 e_1+ \partial_1 m^\Phi_2 e_2+ n\partial_1t_1 e_1+n\partial_1 t_2 e_2 \\
w_2  =\;  & v_2+   \partial_2 m^\Phi_1 e_1+ \partial_2 m^\Phi_2 e_2+     n\partial_2t_1 e_1+n\partial_2 t_2 e_2 
\end{align*}
and then setting $v_1'=v_1+ \partial_1 m^\Phi_1 e_1+ \partial_1 m^\Phi_2 e_2$ and $v_2'=v_2+   \partial_2 m^\Phi_1 e_1+ \partial_2 m^\Phi_2 e_2$ we obtain
\begin{align*}
w_1\wedge w_2  = \;  & v_1'\wedge v_2'\\ 
& +n (\partial_2t_1 v_1'\wedge e_1 + \partial_2 t_2 v_1'\wedge e_2 - \partial_1t_1 v_2'\wedge e_1 -  \partial_1 t_2 v_2'\wedge e_2) \\
& + n^2(\partial_1t_1\partial_2t_2-\partial_1t_2\partial_2t_1)e_1\wedge e_2
\end{align*}
where the dependence on $u$ is implicit. 
The 
monomial $n^{2g}$ appears only in front of $e_1\wedge \cdots \wedge e_{2g}$, and  is multiplied by the function $u\mapsto \det((Dt_f)_u)$. Since the euclidean volume is bounded from above by the integral of the function $u\mapsto \norm{w_1(n;u)\wedge \cdots \wedge w_{2g}(n;u)}$ with respect to the Lebesgue measure on $U\subset \R^{2g}$, this proves the first equivalence stated in the lemma. 

The computation for $\norm{(f^n)^*\kappa^g}_{X_U}$ is similar.
\end{proof}

In the next Section, we extend this type of estimate from compact subset $\overline{U}$ of $B^\circ$ to the whole base  $B$ itself. That is, we shall estimate the volumes of $f^n(S)\subset X$, where $S$ is a multisection.

\section{Propagation of volume estimates} 

 \vspace{0.2cm}
\begin{center}
\begin{minipage}{12cm}
{\sl{ 
We show that if the translation vector of $f$ does not have maximal variation, then 
$\norm{(f^n)^*}_{H^{g,g}(X;\R)}=O(n^{2(g-1)})
$. }}
\end{minipage}
\end{center}
\vspace{0.2cm}

Our first goal is the following proposition. 

\begin{pro}\label{pro:global_volume_estimate}
If the variations of $t_f$ are not maximal then, for any multisection $S$ of $p_f\colon X\to B$, there is 
an integer $D\geq 2$ such that 
$$\vol_\kappa(f^{D^n} (S))=O(D^{2n(g-1)})$$
as $n$ goes to $+\infty$. 
\end{pro}

This will be achieved in Section~\ref{par:proof_proposition_volume_estimate}. Then in Section~\ref{par:cohomology_estimate} we transfer this volume estimate into the upper bound 
$\norm{(f^n)^*}_{H^{g,g}(X;\R)}=O(n^{2g-1})$. 

The difficulty is to propagate the volume estimate from Lemma~\ref{par:local_volume_estimate} up to neighborhoods of the singular fibers of $p_f$ because when approaching these fibers, the Betti coordinates may explode. To do this, we rely on pluripotential theory and use a technique that has been developed by Gauthier and Vigny. 
To refer directly to their work, we translate our problem into a dynamical property of a new (non-invertible, rational) transformation of $X$.

\subsection{Multiplication by $D$ along the fibers} 

\subsubsection{} Let $p\colon X\to B$ be a fibration of a complex projective variety, the generic fiber of which is isomorphic to an abelian variety of dimension $g$. Let $S_0$ be a multisection of $p$ of degree $d=([S_0]\cdot [X_b])$, as in Section~\ref{par:setting_section_betti}.

 Let $D\geq 2$ be an integer such that $d$ divides $D-1$. Then, there is a well defined dominant, rational transformation $m_D\colon X\dasharrow X$ acting by multiplication by $D$ along the smooth fibers of $p$. More precisely, pick a  point $b\in B^\circ$ and a base point $w$ on the fiber $X_b$. Using $w$ as neutral element, $X_b$ becomes a commutative group isomorphic to $\C^{g}/\Lat(b)$ for some lattice $\Lat(b)$.
Using the group law, 
the transformation $m_D$ can be defined fiberwise by
 \begin{equation}
 z\in X_b\mapsto m_D(z)=Dz-\frac{D-1}{d}\sum_{s\in S_0\cap X_b} s,
 \end{equation} 
where the  points of $S_0\cap X_b$ are eventually repeated according to their multiplicities.
This does not depend on 
$w$ by the following standard lemma, the proof of which is straightforward.

\begin{lem} Let $a_i$, $i=1,\dots, l$ be integers such that $\sum_ia_i=1$, and $P_i$, $i=1,\dots, l$ be the points on a complex torus. Then $\sum_i a_iP_i$ does not depend on the choice of the neutral element $w$.  
\end{lem}

Hence it defines a rational transformation of $X$, regular above $B^\circ$, preserving $p_f$, and of topological degree~$D^{2g}$.
In Betti coordinates, above some open subset $U\subset B^\circ$, $m_D$ becomes 
\begin{equation}
m_{D,\Phi}(u,x)=(u, Dx+ t_D(x))
\end{equation} 
for some real analytic map $t_D\colon U\to \R^{2g}$. 

\subsubsection{}\label{par:base-change} If $p\colon X\to B$ is a lagrangian fibration of a hyperkähler manifold, the natural polarization of the fibers introduced in Section~\ref{par:lagrangian-fibration} is automatically invariant under $m_D$; that is,  
\begin{enumerate}[\rm (1)]
\item for $b$ in $B^\circ$, the image of the restriction of $ H^2(X;\Z)$ to $ H^2(X_b;\Z)$ is an infinite cyclic subgroup 
$R_b\subset H^2(X_b;\Z)$;
\item there
is an ample line bundle $A$ on $X$ such that (a suitable multiple of) the ample generator of $R_b$ is the Chern class of $A_b:=A_{\vert X_b}$ and 
$m_D^*A_b=A_b^{\otimes D^2}$. 
\end{enumerate}
Thus, $m_D$ is a family of polarized endomorphisms of $X\to B$ in the sense used by Gauthier and Vigny in~\cite{GV}.  

To prove (2), let $H_b$ be the ample generator of $R_b$. The inverse image by $m_D$ multiplies $H_b$ by $D^2$. If we pick any line bundle $L_b$ in the class $H_b$, then $m_D^*L_b=L_b^{\otimes D^2}\otimes M_b$ where $M_b\in Pic^{\circ}(X_b)$. To find a line bundle in the class of $H_b$, which is taken to the power $D^2$ by $m_D$, 
we have to add to $L_b$ a $(D^2-1)$-th root of $M_b$. There are $D^2-1$ of them, so equally $D^2-1$ line bundles $L_{b,j}$ in the class $H_b$, and we take the sum of them all to get a monodromy invariant $A_b$.

Now, suppose we start with a lagrangian fibration and a multisection $S_0$, and we do the base change given by $p\colon S_0\to B$. We  get a new variety $Y$, a map $q\colon Y\to X$ of degree $d$, and a new fibration $p_Y\colon Y\to S_0$ such that  $p\circ q=p_{\vert S_0}\circ p_Y$. Moreover, $p_Y$ has a natural section $S_0^Y\subset Y$. Then, for each $D\geq 2$ we can construct a rational transformation $m_D^Y\colon Y\dasharrow Y$ acting  by multiplication by $D$ on the smooth fibers of $p_Y$ and fixing $S_0^Y$ pointwise (we use $S_0^Y\cap Y_b$ as the neutral element of  $Y_b:=p_Y^{-1}(b)$ for $b\in S_0^\circ$). The natural polarization $R_b$ of the fibers of $X_b$ can be pulled back to $Y$ and it gives an $m_D^Y$-invariant polarization. 

\subsection{Local to global volume estimates}\label{par:GV} Let us summarize some of the results of~\cite{GV}. We fix a fibration $p_Y\colon Y\to B_Y$ and a rational transformation $g\colon Y\to Y$.
We also fix a Kähler form $\kappa_B$ on $B_Y$.
 We assume that $p_Y\circ g= p_Y$, that $g$ is regular over some dense open subset of $B_Y$, and that, as above, $g$ is (relatively) polarized. This assumption is equivalent to the existence of  an ample line bundle $A$ on $Y$ and an integer $D(g)\geq 2$ such that 
\begin{equation}\label{eq:polarized}
g^*A_{b}=A_{b}^{D(g)}
\end{equation} 
for all $b$ in a dense open subset of $B_Y$. Choose a Kähler form $\kappa$ representing the Chern class of $A$ and a dense open subset $B_Y'$ of $B_Y^\circ$ over which $g$ is regular and satisfies~\eqref{eq:polarized} ($B_Y'$ is a {\emph{regular part}} in the sense of~\cite{GV}). Set $Y'=p_Y^{-1}(B_Y')$. Then, on $Y'$, there is a closed positive current $\hat{T}_g$, of type $(1,1)$, such that 
\begin{equation}\label{eq:fiber_current}
\frac{1}{D(g)^k}(g^k)^*\kappa\to \hat{T}_g
\end{equation}
in the sense of weak convergence for currents. Moreover, $\hat{T}_g$ has local, continuous potentials (on the open set $Y'$). For this, we refer to Section~2.3 of~\cite{GV}~(\footnote{The construction in~\cite{GV} differs slightly from what we write. They fix an equivariant embedding $\iota\colon Y\to B_Y\times \P^N$ such that $\pi=\pi_B\circ \iota$ where $\pi_B$ is the first projection $B_Y\times Y\to B_Y$. Then, they replace the Kähler form $\kappa$ by $\kappa_{FS}$, the restriction of the Fubini-Study form to $\iota(Y)$. So, their form is not Kähler, but in the limit process~\eqref{eq:fiber_current} we obtain the same current.   }).

Then, Proposition~3.3 in \cite{GV} shows that the following properties are equivalent. Let $S$ be a multisection of $p_Y$ and let $\{S\}$ denote the current of integration on~$S$. Let $b_Y$ denote the dimension of $B_Y$, hence also the dimension of $S$. The following properties are equivalent 
\begin{enumerate}[\rm (a)]
\item locally above $B_Y'$ the volume of $g^k(S)$ does not grow as fast as $D(g)^{b_Yk}$. This means that for any open subset $U\subset B_Y$ such that $\overline{U}\subset B_Y'$ we have 
$$
\vol(g^k(S_U))=o(D(g)^{b_Yk})
$$
or equivalently
$$
\int_{g^k(S_U)}\kappa^{b_Y}=o(D(g)^{b_Yk})
$$
as $k$ goes to $+\infty$;
\item the intersection of $\hat{T}_g^{b_Y}$ with $S$ over $B_Y'$ vanishes, i.e. 
$$
\int_{Y'} S\wedge  \hat{T}_g^{b_Y}=0;
$$
\item the global volume of the strict transform $g^k_*(S)$ grows as most as  $D(g)^{(b_Y-1)k}$, i.e.\  
$$ 
\vol(g^K_*(S))=O(D(g)^{(b_Y-1)k})
$$ 
as $k$ goes to $+\infty$; equivalently
$$
\parallel [g^k_*(S)]\parallel =O(D(g)^{(b_Y-1)k}).
$$
\end{enumerate}
Here, $\kappa$ is any Kähler form on $Y$ and $[\cdot]$ denotes the class in $H^{2b_Y}(Y;\Z)$. 
The point is that the local Property~(a), in which the implicit constant in $o(\cdot)$ might depend on $U$, gives rise to the global estimates stated in Property~(c).

If we apply this result when $b_Y=g$, $\dim(Y)=2g$, the generic fiber of  $Y\to B_Y$  is abelian, and $g=m_D^Y$ acts by multiplication by $D$ along the fibers, then $D(g)=D^2$ and we derive that
$\vol(g^k(S_U))=o(D^{2gk})$ implies 
\begin{equation}
\parallel [(m_D^Y)^k_*(S)]\parallel =O(D^{(2g-2)k}).
\end{equation}

\subsection{Proof of Proposition~\ref{pro:global_volume_estimate}}\label{par:proof_proposition_volume_estimate}
Set $X^\circ=X_{B^\circ}$. 

\subsubsection{} First, we assume $p_f\colon X\to B$ has a section $s_0\colon B\dasharrow X$, the image of which is denoted by $S_0$. 

(a). {\emph{First Step}.--} We fix some integer $D\geq 2$ and denote by $m_D\colon X\dasharrow X$ the rational map fixing $S_0$ and acting by multiplication by $D$ along the fibers of $p_f$. Let $\Phi$ denote local Betti coordinates associated to an open set $U\subset B^\circ$, the section $S_0$, and some basis of $H_1(X_b;\Z)$, $b\in U$. Viewed in the Betti coordinates, the section $s_0$ corresponds to $s_{0,\Phi}(u)=0$, while $f$ and $m_D$ correspond to  
\begin{align}
f_\Phi(u,x) &= (u,x+t_f(u))\\
m_{D,\Phi}(u,x) &= (u, Dx).
\end{align} 
Thus, we have 
\begin{align}
f_{\Phi}^{D^k}\circ s_{0,\Phi}(u) &= D^kt_f(u)\\
&= m_{D,\Phi}^k\left( f_\Phi\circ s_{0,\Phi}(u) \right)
\end{align}
which means that above $U$ we have 
\begin{align}
f^{D^{k}}(S_{0,U}) = m_D^k\left( f(S_{0,U})\right).
\end{align}
From Lemma~\ref{par:local_volume_estimate} and Section~\ref{par:GV}, we deduce that 
\begin{align}
\parallel f^{D^k}(S_0)\parallel =O(D^{(2g-2)k})
\end{align}
where $\parallel \cdot \parallel$ is any norm on the vector space $H^{2g}(X;\R)$.

(b). {\emph{Second Step}.--} Consider the fiber product $W=X\times_B X \times_B X$; as a set, this algebraic variety is 
\begin{equation}
W=\{(x,y,z)\in X^3\; ; \; p_f(x)=p_f(y)=p_f(z)\}.
\end{equation}
It comes with a fibration $P_f\colon W\to B$, defined by $P_f(x,y,z)=p_f(x)$, and with a rational map $A\colon W\to X$, defined by
\begin{equation}\label{eq:addition-map}
A(x,y,z)=x+(y-z),
\end{equation}
using the group law along the fiber $X_b$, $b=P_f(x,y,z)$. (We do not need the existence of a section or a choice of neutral element in $X_b$ for this definition.)

Let $S$ be a multisection of $p_f$ and consider the sequence of multisections $T_k\subset W$ of $P_f$ defined by 
\begin{equation}
T_k=\{(x,y,z)\in W\; ; \; x\in S, \; y\in S_0, \; z=f^{D^k}(y)\}.
\end{equation}
Then, 
\begin{equation}
f^{D^{k}}(S)=A(T_k).
\end{equation}
From the first step, we know that the class  $[T_k]\in H^{2g}(W;\Z)$ satisfies 
\begin{equation}
\parallel [T_k] \parallel \leq C D^{(2g-2)k}
\end{equation}
for some constant $C>0$ that depends on $S$. At the level of (co)homology classes, $A$ acts as a linear map between finite dimensional spaces; thus, we obtain $\parallel [A(T_k)]\parallel \leq C'\parallel [T_k]\parallel$ for some constant $C'>0$ and 
\begin{equation}
\parallel [f^{D^{k}}(S)] \parallel \leq C'' D^{(2g-2)k}
\end{equation}
for $C''=C\cdot C'$. This concludes the proof.

\subsubsection{} In case $p_f\colon X \to B$ does not have a section, we take for $S_0$ a multisection and do the base change $p_f\colon S_0\to B$, as in Section~\ref{par:base-change}. This provides a new variety 
$Y\to S_0$ and a map $q\colon Y\to X$ above $S_0\to B$. 
Define $S_0'$ to be the locus of points $s\in S_0$ around which $p_F\colon S_0\to B$ is a local diffeomorphism and $p_f(x)\in B^\circ$. Then, if $V\subset S_0'$ is an open subset and $V$ is small enough, $q$ realizes a diffeomorphism from  $Y_V=p_Y^{-1}(V)$ to $p_f^{-1}(p_f(V))$. 
The automorphism $f$ induces a rational transformation $f_Y$ of $Y$ such that $f\circ q = q\circ f_Y$. For $D\geq 2$, we define $m_{D,Y}$ to be the multiplication by $D$ along the fibers of $p_Y$ fixing the natural section $S_0^Y$. 

Now, if $S$ is a multisection of $p_f$, we pull back it to $Y$ by $q$. This gives a multisection $S^Y$ of $p_Y$ for which  $\vol (f_Y^{D^k}(S^Y_U))=O(D^{(2g-1)k})$ as soon as $U\subset S_0'$, because $q$ realizes a local conjugation between $f_Y$ and $f$. Thus, we can repeat the argument from Step 1 above in $Y$. Then, we can repeat the argument from Step 2 by working on $W_Y=Y\times_{S_0} Y \times_{S_0} Y$ and composing the addition map $A\colon W_Y\to Y$ 
(defined as in Equation~\eqref{eq:addition-map}) with $q\colon Y\to X$. This concludes the proof in the general case. 

\subsection{A cohomological estimate}\label{par:cohomology_estimate} 

\begin{pro}\label{pro:upper_cohomological_estimate}
Let $f$ be a parabolic automorphism of a hyperkähler manifold $X$.
If the variations of the  translation vector of $f$ are not maximal, then 
$$
\norm{(f^n)^*}_{H^{g,g}(X;\R)}=O(n^{2(g-1)})
$$
as $n$ goes to $+\infty$.
\end{pro}

\begin{proof}
Embed $X$ into a projective space $\P^N$, and intersect it with a suitable linear subspace to get a multisection $S$ of $X$. The class $[S]$ of $S$ will be considered as an element of $H^{g,g}(X;\Z)$ (using Poincaré duality), it is the same as $\omega_{FS}^g$ where $\omega_{FS}$ denotes the restriction of the Fubini-Study form to $X$. Proposition~\ref{pro:global_volume_estimate} shows that 
\begin{equation}\label{eq:cohom_estimate_S}
\norm{(f^n)^*[S]}=O(n^{2(g-1)})
\end{equation} 
along the subsequence $n=D^k$, for some $D\geq 2$. 

Now, consider a real  subspace $W$ of $H^{2g}(X;\R)$ together with a closed, convex, and salient cone $C\subset W$, the interior of which is non-empty. Assume that (a) $W$ and  $C$ are $f^*$-invariant and  (b) $[S]$ is in the interior of $C$. Then, $\norm{(f^*)^n}_{W}=O(n^{2(g-1)})$. Indeed, by Birkhoff's version of the Perron-Frobenius theorem, we know that 
\begin{equation}
A^{-1} \norm{(f^*)^n[S]} \leq \norm{(f^*)^n}_W \leq A \norm{(f^*)^n[S]}
\end{equation}
for some constant $A\geq 1$. Thus, Equation~\eqref{eq:cohom_estimate_S} 
implies that the spectral radius of $f^*_W$ is equal to $1$. 
This, in turn, implies that $ \norm{(f^*)^n}_W$ grows like a power of $n$, and then Equation~\eqref{eq:cohom_estimate_S} shows that this power is $\leq 2(g-1)$.

We apply this scheme to the pair $P^g(X)\subset H^{p,p}(X;\R)$ where $P^g(X)$ is the 
cone of classes represented by closed positive currents of bidegree $(p,p)$ (see for instance ~\cite{demailly:book}). 
This is a closed, convex cone, and it is salient because the set of closed positive currents $T$ of bidegree $(g,g)$ with fixed mass $M_\kappa(T)=\langle T \vert \kappa^g\rangle$ is compact for the weak-$*$ topology. The class $[S]=[\omega_{FS}^g]$ is in the interior of this cone, because a small perturbation of $\omega^g$ is a positive $(g,g)$-form.  And $P^g(X)$ is $\Aut(X)$-invariant. This concludes the proof. 
\end{proof}

\section{Action on the cohomology} 


 \vspace{0.2cm}
\begin{center}
\begin{minipage}{12cm}
{\sl{
In this section, we show that if $f$ is a parabolic automorphism of a hyperkähler manifold, 
then  for every $1\leq p\leq g$ there is a positive constant $c_p(f)$ such that $\norm{(f^*)^n}_{H^{p,p}(X;\R)}\simeq c_p(f) n^{2p}$. }}
\end{minipage}
\end{center}
\vspace{0.2cm}

Let $f$ be an automorphism of a hyperkähler manifold. Recall from Section~\ref{par:loxo_para_elli} that $f$ is either elliptic, parabolic, or loxodromic. 
The following result is a concatenation of theorems of Lo Bianco, Oguiso, and Verbitsky. 

\vv

\begin{thm-C} 
Let $f$ be an automorphism of a hyperkähler manifold of dimension $2g$. 

\begin{enumerate}[\rm (1)]
\item If $f$ is elliptic, then $f^k=\Id_X$ for some $k\geq 1$;
\item if $f$ is parabolic, then for every $1\leq p\leq g$ there is a positive 
constant $c_p(f)$ such that 
$\norm{(f^*)^n}_{H^{p,p}(X;\R)}\simeq c_p(f) n^{2p};$
\item  if $f$ is loxodromic, there is a real number $\lambda(f)>1$ such that for every $1\leq p\leq g$
there is a positive constant $c_p(f)$ such that 
$\norm{(f^*)^n}_{H^{p,p}(X;\R)}\simeq c_p(f) \lambda(f)^{pn};$
\item by duality, if $f$ is parabolic then for $g\leq p\leq 2g$ there is a positive constant $c_p(f)$ with $\norm{(f^*)^n}_{H^{p,p}(X;\R)}\simeq c_p(f) n^{2(2g-p)}$, and similarly in the loxodromic case.
\end{enumerate}
\end{thm-C}

\vv

We only sketch the proof, because the only part that may be considered to be new is the second assertion.

To prove Assertion~(1), note that the connected component of the identity in $\Aut(X)$
is trivial (see~\cite{huybrechts:survey2003}). Hence, by Lieberman's theorem (see~\cite{Lieberman:1978}), an automorphism preserving a Kähler class has finite order. But if $f$ is elliptic, its eigenvalues on $H^{1,1}(X;\R)$, and then on 
$H^{2}(X;\R)$, all have modulus~$1$. Thus, being roots of the characteristic polynomial of $f^*_{H^2(X;\Z)}$, hence of a monic polynomial with integer coefficients, these eigenvalues must be roots of unity. This implies that a positive iterate of $f$ acts trivially on $H^2(X;\Z)$, and we conclude with Lieberman's theorem.

Now, suppose that $f$ is parabolic. Then $\norm{(f^*)^n}_{H^{1,1}(X;\R)}\simeq c_1(f) n^{2}$ for some $c_1(f)>0$ (see \S~\ref{appendix})
And a theorem of Verbitsky says that  $\Sym^p(H^2(X;\R))$ embeds into $H^{2p}(X;\R)$ for $p\leq g$ via the cup product, so 
\begin{equation}\label{eq:verbitsky_lower_estimate}
\norm{(f^*)^n}_{H^{p,p}(X;\R)}\geq  c n^{2p}
\end{equation}
for some positive constant $c$ and for all $n\geq 1$. On the other hand, from the Khovanski-Teyssier inequalities,  the numbers 
\begin{equation}
s_p(f)=\limsup_{n\to +\infty}\frac{\log(\norm{(f^*)^n}_{H^{p,p}(X;\R)})}{\log(n)}
\end{equation}
form a concave function of $p$, for $0\leq p \leq \dim(X)$ (see~\cite{lobianco:bsmf}); this means that
\begin{equation}
2s_p(f)\leq s_{p-1}(f)+s_{p+1}(f)
\end{equation} for every $p$. Since $s_0(f)=0$ and $s_1(f)=2$, it follows from the lower estimate~\eqref{eq:verbitsky_lower_estimate} that $s_p(f)=2p$ for every $p\leq g$.

When $f$ is loxodromic, we have $\norm{(f^*)^n}_{H^{1,1}(X;\R)}\simeq c_1(f) \lambda(f)^{n}$ for some $\lambda(f)>1$ and the same  argument leads to Assertion~(3) (for the details, see~\cite[Appendix]{lobianco:imrn} or \cite{oguiso}).

\section{Proof of theorem A and application} 

 \vspace{0.2cm}
\begin{center}
\begin{minipage}{12cm}
{\sl{
We prove Theorem~A when $X$ is assumed to be projective and give an application to the dynamics of some groups of automorphisms.
 }}
\end{minipage}
\end{center}
\vspace{0.2cm}

\subsection{Proof of Theorem~A}\label{par:proof-thmA-projective} Let $f$ be a parabolic automorphism of a hyperkähler manifold $X$, as in Theorem~A.
Assertion~(1) of Theorem~A is contained in Theorem~C.
In particular, $\norm{(f^*)^n}_{H^{g,g}(X;\R)}$ grows like $c_g(f)n^{2g}$. The same estimate holds for $f^k$, the first positive iterate acting as the identity on the base of the invariant lagrangian fibration $p_f$, with constant $c_g(f)$ replaced by $c_g(f^k)=c_g(f)k^{2g}$. 

Now, assume that $X$ is projective. By Proposition~\ref{pro:upper_cohomological_estimate}, the translation vector of $f^k$ has maximal variations, as stated in Assertion~ (2) of Theorem~A. 

Let us derive Assertion~(3) from Assertion~(2). This final step does not use that $X$ is projective, only the validity of Assertion~(2).  
Let $U\subset B^\circ$ be a small, relatively compact, simply connected open subset. Write $f^k$ on $X_U$ in some Betti coordinates : 
\begin{equation}
f^k_\Phi(u,x)=(u,x+t_{f^k}(u)).
\end{equation}
For $b\in U$, the closure $Z(b)$ of $\Z t_{f^k}(b)\in \R^{2g}/\Z^{2g}$ is a Lie subgroup of $\R^{2g}/\Z^{2g}$. 
Its dimension $r(b)$ and its number of connected components $c(b)$ vary with~$b$. 
Since $t_{f^k}$ has maximal variations, 
$t_{f^k}(U)$ contains an open set, so that any pair $(r,c)$ with $0\leq r \leq 2g$ and $c\in \Z_{\geq 1}$ can be realized as $(r(b),c(b))$ by some $b$ in $U$. This proves Assertion~(3) because orbits of $f^k$ in $X_b$ correspond, in the Betti coordinates $U\times \R^{2g}/\Z^{2g}$, to subsets of type $\{b\}\times (x+Z(b))$.

\subsection{Example} Consider a compact hyperkähler manifold $X$, of dimension $2g$, together with two parabolic automorphisms $f$ and $g$ such that the Lagrangian fibrations $p_f\colon X\to B_f$ and $p_g\colon X\to B_g$ are distinct. By this we mean that $p_f$ and $p_g$ satisfy the following equivalent properties: (a) the restriction of $p_f$ to a general fiber of $p_g$ is not constant;
(b) the restriction of $p_f$ to a general fiber of $p_g$ is a dominant morphism onto~$B_f$;
(c) the restriction of $p_g$ to a general fiber of $p_f$ is a dominant morphism onto~$B_g$.
The existence of two distinct Lagrangian fibrations like that implies that $X$ is projective.

For each pair of positive integers $(k,\ell)$, consider the subgroup $\Gamma_{k,\ell}$ of $\Aut(X)$ generated by $f^k$ and $g^l$. Fix a distance $\dist(\cdot)$ on $X$ and say that a subset $\Lambda\subset X$ is $\epsilon$-dense if every point of $X$ is at distance less than $\epsilon $ from a point of $\Lambda$. 
\begin{enumerate}[\rm (1)]
\item The set  of points $x\in X$ such that the orbit $\Gamma_{k,\ell}(x)$ is dense in $X$ (for the euclidean topology) and is a countable intersection of open dense subsets of $X$; in particular, it is dense and has full measure for the volume form $(\sigma\wedge{\overline{\sigma}})^g$.
\item For every $\epsilon>0$, there is an integer $N\geq 1$ such that for every pair $(k,\ell)$ with $k$ and $\ell\geq N$, the set of points $y\in X$ such that $\Gamma_{k,\ell}(y)$ is finite is an $\epsilon$-dense set.  
\end{enumerate} 
To get (1), do as in~\cite[\S 6]{AV} and~\cite{Cantat-Dujardin:Transformation-Groups}: by Theorem~A, the locus of points $b$ in $B_f$ such  that the orbit closures of $f$ in $X_b$ have codimension $\geq 1$ is a countable union of proper real analytic subsets of $B_f^\circ$. Pick a point 
$x\in X$ such that $p_f(x)$ is not in this meager set. The closure of $\Gamma_{k,\ell}(x)$ contains the fiber of $p_f$ through $x$, hence also the orbit of this fiber by $\Gamma_{k,\ell}$; then, using $g$, it contains all fibers of $p_g$ in the complement of a meager set. This shows that a generic orbit is dense. To conclude, note that being dense is the same as being $\epsilon$-dense for all $\epsilon>0$. (Also, one can notice that an orbit is dense for some $\Gamma_{k,\ell}$ if and only if it is dense for all $\Gamma_{m,n}$.)

To get Assertion~(2), we use the following consequence of Theorem~A: the set $F_k:=\{b\in B_f\; ; \; f^k_{\vert X_b}=\Id_{X_b}\}$ becomes $\epsilon$-dense in $B$ if $k$ is large enough. Similarly, the set $G_\ell\subset B_g$ corresponding to fibers of $p_g$ on which the order of $g$ divides $\ell$ is $\epsilon$-dense for large enough $\ell$. Then, the set 
\begin{equation}
p_f^{-1}(F_k)\cap p_g^{-1}(G_\ell)
\end{equation}
is $\epsilon$-dense and is made of fixed points of $\Gamma_{k,\ell}$.

\begin{rem}
The second assertion shows the finite exceptional orbits in Theorem 0.2 of~\cite{Dolce-Tropeano} can be arbitrarily large (resp. $\epsilon$-dense) 
when one reduces the size of the group (see also~Corollary 1.2 of~\cite{CTZ}).
\end{rem}
\section{From the projective case to the Kähler case} 

 \vspace{0.2cm}
\begin{center}
\begin{minipage}{12cm}
{\sl{
The purpose of this section is to deduce Theorem~A in the non-projective setting from the case of projective hyperkähler manifolds.
For this, we apply the method of degenerate twistor deformations developed by Verbitsky   and  Soldatenkov  in \cite{verbitsky:degenerate, soldatenkov-verbitsky}.
 }}
\end{minipage}
\end{center}
\vspace{0.2cm}

According to Section~\ref{par:proof-thmA-projective}, we only need to prove the second assertion of Theorem~A, since the first one has already been verified for all hyperkähler manifolds and the third one is a consequence of the second.

\subsection{The non-projective setting}\label{par:non-projective-setting}
Let $X$ be a {\sl{non-projective}}, irreducible, hyperk\"ahler manifold,   with a fixed holomorphic structure, equipped with some holomorphic lagrangian fibration $p\colon X\to B$.  Let $\sigma$ be a holomorphic symplectic form on $X$.
Recall from Section~\ref{par:lagrangian-fibration} that $B$ is projective.
We denote by $h$ the pull-back of a fixed very ample class in $\NS(B)$. 
The class $h\in \NS(X)$ is nef and satisfies $q_X(h)=0$.

Since $X$ is not projective, Huybrechts' Theorem (see \S~\ref{par:neron-severi-group}) shows that  the form $q_X$ is negative semi-definite on $\NS(X;\R)$, with one-dimensional kernel $\R h$. In  other words, with the vocabulary from Sections~\ref{par:neron-severi-group} and~\ref{par:transcendental_lattice}, the lattice $(\NS(X), q_X)$ is parabolic and the class $h$ is an element of both $\NS(X)$ and the transcendental lattice $\T(X)$.

Consider a parabolic automorphism $f\colon X\to X$ such that $p\circ f=f_B\circ p$ for some automorphism $f_B$ of $B$. By definition, $f^*h=h$. Moreover, according to 
Section~\ref{par:automorphisms},  there is a Kähler form $\kappa_B$ on $B$ 
such that $f_B^*\kappa_B=\kappa_B$; we can  assume furthermore that $[\kappa_B]=h$. 


In the non-projective case, it follows from Oguiso's results that $f^*\sigma=\sigma$ (see Assertion (1) of Theorem~B and the comment after this theorem). 

\subsection{Twistor deformations} 
{\sl{A {\sc{c}}-symplectic form on a differentiable ma\-nifold $M$ of dimension $4n$ is a closed, complex-valued $2$-form $\Omega$ such that $\Omega^{n+1}=0$ and $\Omega^{n}\wedge \overline {\Omega^{n}}$ is everywhere non-vanishing}}  (see Definition 2.1 in \cite{soldatenkov-verbitsky}).  The first main properties of such a {\sc{c}}-symplectic form are (see \cite{soldatenkov-verbitsky}):
\begin{enumerate}[\rm (1)]
\item the kernel of $\Omega$ on the complexified tangent space $T_\C M$ is everywhere of rank $2n$ and can be seen as the antiholomorphic tangent bundle of a complex structure $\J$; 
\item with respect to this complex structure, $\Omega$ is a holomorphic symplectic form.
\end{enumerate}
Take $X$ as in~Section~\ref{par:non-projective-setting}. On the differentiable manifold $X$, consider the family of differentiable forms $\Omega_t=\sigma + tp^*\kappa_B$. According to  \cite[Theorem 2.3]{soldatenkov-verbitsky}, 
\begin{enumerate}[\rm (1)]
\item[(3)]  for any $t$, $\Omega_t$ is a {\sc{c}}-symplectic form. 
\end{enumerate}
Thus, the family $\Omega_t$ defines a family of complex structures $\J_t$, hence a family of complex manifolds (on the same  underlying differentiable manifold) such that $X_0=X$, $\Omega_t$ is a holomorphic symplectic form on $X_t$, and therefore $H^{2,0}(X_t)$ is generated by $[\Omega_t]=[\sigma] + th$ for each $t$. Note, moreover, that 
\begin{enumerate}[\rm (1)]
\item[(4)] the map $p:X_t\to B$ remains a holomorphic lagrangian fibration, and the complex structure on the fibers of $p$ does not change. \end{enumerate} 
Indeed, $p^*\kappa_B$ vanishes identically along each fiber of $p$.

\subsection{Conclusion}   
Now, the next proposition is almost obvious. 

\begin{pro} The diffeomorphism $f:X\to X$ is holomorphic with respect to all the complex structures $\J_t$, $t\in \C$.
\end{pro}

\begin{proof} First, recall that we assume that the parabolic automorphism is symplectic with respect to the initial symplectic structure $\sigma=\Omega_0$. Moreover, $f_B$ preserves $\kappa_B$. Thus, $f$ 
preserves each of {\sc{c}}-symplectic forms $\Omega_t=\sigma+tp^*\kappa_B$.  In particular, $f$ preserves its kernel and the conjugate of the kernel, which are the antiholomorphic and holomorphic tangent bundles to $X_t$, respectively.\end{proof}

With this construction at hand, the keypoint is that some of the $X_t$ are projective for arbitrary small parameters $t\in \C$. 

\begin{pro} For any $r>0$, there  exists $t\in \C$ with $\vert t\vert <1/r$ such that   $X_t$ is projective.
\end{pro}

\begin{proof} For $t$ small, $X_t$ is Kähler, because being Kähler is an open property (see~\cite{voisin:book-hodge}). Thus, for small parameters, we obtain a family of irreducible hyperkähler manifolds (indeed the Hodge numbers are constant in families of K\"ahler manifolds). 

From Huybrechts' Theorem, we know that $X_t$ is projective when it carries an integral $(1,1)$-class $u$ with $q(u,u)>0$. Moreover,  an integral class $u\in H^2(X;\Z)$ is 
of type $(1,1)$ with respect to the complex structure $\J_t$ when it is $q$-orthogonal to the class $[\Omega_t]$.

Fix $r\geq 1$ large, and then choose a class 
$a\in H^2(X,\Z)$ such that 
\begin{equation}
0 < rq(a,\sigma) < q(a,h);
\end{equation} 
such a class exist because $q(\cdot, \sigma)$ and $q(\cdot, h)$ are two linearly independant linear forms on $H^2(X;\C)$. 
Now, $a$ is of type $(1,1)$ on $X_t$ if and only if $t=-q(a,\sigma)/q(a,h)$. This defines a unique $t$, of modulus $<1/r$. \end{proof}

 Since Assertion (2) of Theorem~A does not depend on the complex structure $\J_t$ but only on the dynamical properties of $f$, we can now apply Theorem~A on $X_t$ to derive the same conclusion on $X_0$. This concludes the proof of our main theorem in the non-projective setting.

\subsection{Extension of Lo Bianco's theorem}\label{par:proof-lobianco-non-projective} 
A similar argument can be applied to extend the second assertion of Theorem~B from the projective to the non-projective setting. Indeed, along a twistor deformation $(X_t, \Omega_t)$, with $\Omega_t=\sigma+tp^*\kappa_B$, the action of $f_t$ on the base of its invariant fibration does not change. Since for  some  parameters $t$ we know that $X_t$ is projective, we can apply Lo Bianco's result to conclude that $f_B$ has finite order.


\section{Appendix}\label{appendix} 
{\small{


\subsection{Parabolic isometries} 
Let $V$ be a real vector space of (finite) dimension $m+1$ endowed with a non-degenerate quadratic form $q_V$ of signature $(1,m)$. Let $h\in {\mathsf{O}}(q_V)$ be a linear transformation of $V$ preserving $q_V$. 
By definition, $h$ is parabolic if it does not fix any vector $v\in V$ with $q_V(v)=1$ and if all its eigenvalues have modulus $1$. Equivalently, $1$ is an eigenvalue of $h$, but the corresponding eigenspace $\{v\; ; \; h(v)=v\}$ does not intersect $\{v\; ; \; q_V(v)>0\}$. When $h$ is parabolic, there is a unique isotropic line $D_h\subset V$ which is $h$-invariant, and this line is fixed pointwise.
For this, we refer to Ratcliffe's book on hyperbolic geometry~\cite{Ratcliffe:book}.

\begin{rem}
If $h$ is a parabolic isometry and $h$ preserves a subspace $W\subset V$ on which $q_V$ is non-degenerate and indefinite, then $h_{\vert W}\colon W\to W$ is also parabolic. 

If $\dim(V)\leq 2$, there is no parabolic isometry; indeed, if $\dim(V)=2$ the isotropic cone is made of two lines, and a parabolic isometry should preseve each of them, with eigenvalue $1$ on one of them, hence with eigenvalue $\pm 1$ on the second (because $\det(\cdot)=\pm 1$ on $ {\mathsf{O}}(q_V)$), but then the isometry would have order $1$ or $2$.
\end{rem}

\begin{pro}\label{pro:parabolic_growth}
Let $h$ be such a parabolic isometry. Then, given any operator norm $\norm{\cdot}$ on $\End(V)$, there is a constant $c(h)>0$ such that $\norm{h^n}\simeq c(h) n^2$.
\end{pro}

\begin{proof}
The characteristic polynomial of $h$ can be written $P_h(t)=(t-1)^rQ(t)$ where $Q\in \R[t]$ and $Q(1)\neq 0$. From this, we get a decomposition $V=E_1\oplus E_Q$ where $E_1$ is the kernel 
of $(h-\Id_V)^r$ and $E_Q$ is the kernel of $Q(h)$. This is an orthogonal decomposition $E_1^\perp =E_Q$.  In particular, the restriction of $q_V$ to $E_1$  is non-degenerate. The line $D_h$ is contained in $E_1$. Thus, the signature of $q_V$ on $E_1$ is $(1,\dim(E_1)-1)$, and $q_V$ is negative definite on $E_Q$. In particular, the restriction of $h$ to $E_Q$ is in a compact group.

Thus, we can now assume that $V$ is equal to $E_1$. In other words,  $h$ is unipotent. Then, on $D_h^\perp/D_h$, the endomorphism induced by $h$ is unipotent and preserves a negative definite quadratic form; thus, it is equal to the identity. 

Let $v_2$ be an element of $V$ such that $q_V(v_2)=1$. Set $v_1=h(v_2)-v_2$. Then $v_1\neq 0$ (because $h$ is parabolic) and $v_1$ is orthogonal to $D_h$. Thus, $v_1=h(v_1)-v_0$ for some $v_0\in D_h$. The vector $v_0$ is not $0$, because otherwise $h$ would induce a parabolic isometry of the $2$-dimensional space ${\mathrm{Vect}}(v_1,v_2)$. Thus, the vector space $W={\mathrm{Vect}}(v_1,v_2,v_3)$ is $h$-invariant, has dimension $3$, and contains $D$. In the basis $(v_1,v_2,v_3)$, the matrix of $h$ is a Jordan bloc of size $3$, and the growth of $\norm{h^n}$
is quaratic. On the orthogonal complement $W^\perp$, $h$ is the identity. This concludes the proof. \end{proof}

\subsection{Parabolic automorphisms}  Combining the two previous sections, we get the notion of parabolic automorphism of hyperkähler manifolds: these are automorphisms $f\colon X\to X$ such that $f^*$ determines a parabolic isometry of $V:=H^{1,1}(X;\R)$ with respect to the  Beauville-Bogomolov quadratic form $q_V:=q$. Then, Proposition~\ref{pro:parabolic_growth} proves the equivalence between  Assertions (a) and (c) from Section~\ref{par:loxo_para_elli}. The equivalence with (b) comes from the fact that all eigenvalues of $f^*$ are roots of unity because $f^*$ preserves the lattice $H^2(X;\Z)$ in $H^2(X;\R)$.

}}

\bibliographystyle{plain}
\bibliography{referencesParabolic}

\end{document}